%% file: SK.tex
\documentclass[11pt,reqno]{amsart}
\usepackage[tt=false]{libertine}
\usepackage{amssymb}
\usepackage[varbb]{newpxmath}

\let\savedbigtimes\bigtimes
\let\bigtimes\relax
\usepackage{mathabx}
\let\bigtimes\savedbigtimes

\usepackage[margin=1in, bottom=1in]{geometry}
\usepackage{enumerate}
\usepackage{graphicx}
\usepackage{subcaption}
\usepackage{enumitem}
\usepackage{multirow}

\usepackage{xpatch}
\xpatchcmd{\paragraph}{\normalfont}{{\normalfont\bfseries}}{}{}

\newtheorem{theorem}{Theorem}[section]
\newtheorem{lemma}[theorem]{Lemma}

\newtheorem{proposition}[theorem]{Proposition}

\newtheorem{definition}[theorem]{Definition}

\newtheorem{remark}[theorem]{Remark}

\usepackage[usenames,dvipsnames]{xcolor}
\usepackage[colorlinks=true,
    linkcolor=teal!60!black,
    citecolor=PineGreen,
    urlcolor=RedViolet]{hyperref}
\hypersetup{bookmarksopen=true}

\usepackage[yyyymmdd,hhmmss]{datetime}

\input{./commands}

\title{Exponentially Slow Mixing of the Low Temperature SK Model}

\author{Mark Sellke}

\begin{document}

% \setcounter{tocdepth}{1}

% \tableofcontents

\begin{abstract}
We give a short proof that low-temperature dynamics for the Sherrington–Kirkpatrick model have mixing time exponential in the system size, based on the recently proved existence of gapped spin configurations \cite{MSS-FB,DGZ-MaxStable}.
This result is in contrast with a well established physics prediction which posits a stretched exponential mixing time of order $e^{N^{1/3 \pm o(1)}}$.
Our proof clarifies that this prediction cannot apply to mixing from worst case initial conditions, but should presumably be understood to concern dynamics from a suitably random initialization.
\end{abstract}

\maketitle

\section{Model and notions of gapped states}
Given a matrix $\bG=(g_{i,j})_{1\le i, j\le N}$ of IID standard Gaussian couplings, the Sherrington--Kirkpatrick (SK) Hamiltonian is defined by
\begin{equation}\label{eq:H}
H_N(\sigma)\coloneqq \la \bsig,\bG\bsig\ra/\sqrt{N},\qquad \bsig\in\bSig_N\equiv \{\pm 1\}^N.
\end{equation}
Fixing an inverse temperature $\beta>0$, the associated Gibbs measure is 
\[
\mu_{\beta}(\sigma;\bG)=e^{\beta H_N(\sigma)}/Z_N(\beta)
\]
where $Z_N(\beta)=\sum_{\bsig\in \{\pm 1\}^N} e^{\beta H_N(\bsig)}$ is the partition function.
Introduced in \cite{sherrington1975solvable}, the SK model is the prototypical example of a mean-field disordered system.
This model is known to exhibit a phase transition at the critical temperature $\beta_c=1$.
For $\beta\leq \beta_c$, the free energy is given by the replica-symmetric prediction
\[
\lim_{N\to\infty}
\bbE\log Z_N(\beta)
=
\lim_{N\to\infty}
\log \bbE Z_N(\beta)
=
\beta^2/2,
\]
while for $\beta>\beta_c$ the limiting free energy is strictly smaller than $\beta^2/2$ and requires the theory of replica symmetry breaking to understand (see e.g. \cite{parisi1979infinite,parisi1980sequence,ruelle1987mathematical,aizenman2003extended,talagrand2006parisi,panchenko2013parisi,panchenko2014parisi,panchenko2013sherrington}).

We will be concerned with the dynamics, which have been extensively studied since \cite{sompolinsky1981dynamic}.
In the high-temperature phase, it is now understood that the canonical Glauber dynamics mixes rapidly for $\beta$ smaller than a positive constant \cite{eldan2021spectral,anari2021entropic,adhikari2022spectral,anari2024trickle}, with efficient approximate sampling being possible for all $\beta<\beta_c$ using alternative diffusion-based methods \cite{alaoui2022sampling,alaoui2025sampling,huang2024sampling}.

The mixing time of the low temperature SK model was extensively studied in the physics literature \cite{rodgers1989distribution,vertechi1989energy,billoire2010distribution}, with (non-rigorous) theoretical arguments and simulations suggesting the stretched exponential scaling
\begin{equation}
\label{eq:t-mix-1/3}
\tmix(\beta)=e^{N^{1/3 \pm o(1)}}
\end{equation}
for the mixing time of Glauber dynamics at inverse temperature $\beta>\beta_c$.
(Note $\tmix$ is itself a random variable, so this prediction is implicitly for its typical behavior.)
Notably, \cite{monthus2009eigenvalue} evaluated the exact spectral gap for several instantiations of the SK model (for system sizes $6\leq N\leq 20$), and found close numerical agreement with \eqref{eq:t-mix-1/3}.

Our main result disproves a strong form of this prediction, showing that the mixing time is exponential in $N$ at low temperature.

\begin{theorem}
\label{thm:main}
There exist absolute constants $\beta_*,c>0$ such that for all $\beta>\beta_*$, the mixing time of Glauber dynamics $\tmix(\beta)$ for \eqref{eq:H} satisfies
\[
\lim_{N\to\infty}
\bbP[\tmix(\beta)\geq e^{c\beta N}]
=1.
\]
\end{theorem}

Theorem~\ref{thm:main} follows straightforwardly from the existence of gapped states recently proved by \cite{MSS-FB,DGZ-MaxStable}, and the main contribution of this note is to point out this implication.
It is worth noting that some of the physics predictions above concern free energy barriers between metastable states, which are apriori different from the worst-case mixing time.
Indeed, gapped states form an exponentially rare portion of the landscape, both according to the uniform measure, and presumably under $\mu_{\beta}$.
Further, \cite{huang2025strong} provided evidence that gapped states are algorithmically intractable to find.
It is thus quite plausible that the prediction \eqref{eq:t-mix-1/3} is correct in some modified form, e.g. for the mixing time from a uniform or Gibbs-distributed initialization.
However the need for this subtlety does not seem to have been observed previously.

We finally note that \cite{alaoui2022sampling} provides evidence of algorithmic hardness for low-temperature sampling from the SK model, via a transport disorder chaos property that implies failure of all stable sampling algorithms. 
This class includes gradient-based methods such as diffusion sampling, but seemingly not Glauber dynamics run for super-logarithmic $\omega(\log N)$ time.
Theorem~\ref{thm:main} thus constitutes further evidence for the computational intractability of sampling at low temperature.

\section{Proof of Theorem~\ref{thm:main}}

For a configuration $\sigma$ and site $i$, the local field is defined to be
\begin{equation}
\label{eq:local-field}
L_i(\bsig)
=
\partial_{\sigma_i} H_N(\bsig)
=
\sigma_i\cdot (H_N(\bsig)
-
H_N(\bsig\oplus e_i)) / 2
.
\end{equation}
Here $\bsig\oplus e_i\in\bSig_N$ disagrees with $\bsig$ precisely in the $i$-th coordinate.
By definition, any local maximum for the SK model (with respect to single spin flips) satisfies $\sigma_i L_i(\bsig)\geq 0$ for all $i\in [N]$.
We say such a local maximum is a $\gamma$-gapped state if $\min_i \sigma_i L_i(\bsig)\geq \gamma>0$, i.e. the energy cost of any spin flip is uniformly positive.
The existence of gapped states was proved for the case of Rademacher disorder by \cite{MSS-FB}.
To deduce Theorem~\ref{thm:main}, a slightly relaxed notion of gapped state suffices, which was shown in \cite{DGZ-MaxStable} to exist under extremely general conditions on the disorder.

\begin{definition}
\label{def:gapped}
    For $\gamma,\delta>0$, the point $\bsig^*\in\bSig_N$ is a \textbf{$(\gamma,\delta)$-gapped state} for $H_N$ if
    \[
    |\{i\in [N]~:~\sigma^*_i L_i(\bsig^*)<\gamma\}|\leq \delta N.
    \]
\end{definition}

\begin{proposition}[{\cite[Theorem 3]{DGZ-MaxStable}}]
\label{prop:DGZ}
Let $(g_{i,j})_{1\leq i<j\leq N}$ be IID random variables with mean $0$, variance $1$, and uniformly bounded third moment.
Then for all $\gamma\in (0,\gamma_*)$ and arbitrarily small $\delta>0$, with high probability as $N\to\infty$, there exists a $(\gamma,\delta)$-gapped state $\bsig^*$ for $H_N$.
\end{proposition}

We will deduce Theorem~\ref{thm:main} by using such gapped states as bottlenecks for the low-temperature dynamics.
The idea is simply that on a Hamming sphere around any gapped state, the energy must be uniformly smaller than at the gapped state itself.
The definition of gapped state implies a non-trivial first order Taylor approximation for the energy difference. To control the second order term we use the following simple estimate.
Here we set 
\[
\bG^{\sym}=\frac{\bG+\bG^{\top}}{\sqrt{N}}
\]
which has the law of a $\mathsf{GOE}(N)$ matrix multiplied by $\sqrt{2}$.
We also let $\bG^{\sym}_{I\times I}$ be its restriction to $\bbR^I\times \bbR^I$.

\begin{proposition}[{\cite[Proposition 2.2]{huang2025strong}}]
\label{prop:restricted-norm-bound}
There is a constant $C_{\ref{prop:restricted-norm-bound}} > 0$ such that the following holds.
For any $\rho\in (0,1/2)$ and $N$ sufficiently large, the matrix $\bG^{\sym}$ satisfies with probability $1-e^{-cN}$:
\[
    \sup_{|I|\leq \rho N}
    \|\bG^{\sym}_{I\times I}\|_{\op}
    \leq
    C_{\ref{prop:restricted-norm-bound}} \sqrt{\rho\log(1/\rho) N}.
\]
\end{proposition}

We fix $\gamma>0$ and choose $\rho>0$ small depending on $\gamma$, and $\delta>0$ small depending on $(\gamma,\rho)$. 

\begin{lemma}
\label{lem:sphere-energies-smaller}
Suppose Proposition~\ref{prop:restricted-norm-bound} holds for the value $\rho$, and $\|\bG^{\sym}\|_{\op}\leq 3$.
Let $\bsig^*$ be a $(\gamma,\delta)$-gapped state.
Then for any $\bsig\in\bSig_N$ with $\|\bsig-\bsig^*\|_{L^1}=\rho N$, we have
\[
H_N(\bsig^*)-H_N(\bsig)
\geq 
\rho\gamma N/2.
\]
\end{lemma}

\begin{proof}
    Since $H_N(\bsig)=\la \bsig, \bG^{\sym}\bsig\ra/\sqrt{N}=\la \bsig, \bG^{\sym}\bsig\ra/\sqrt{2}$ is a homogenous quadratic function, we have 
    \[
    H_N(\bsig^*)-H_N(\bsig)
    =
    \la
    \nabla H_N(\bsig^*)
    ,
    \bsig^*-\bsig
    \ra 
    -
    \la (\bsig^*-\bsig),\bG^{\sym}(\bsig^*-\bsig)\ra/\sqrt{2}.
    \]
    For the first term, at least $(\rho-\delta) N$ coordinates contribute at least $\gamma$ to the inner product by Definition~\ref{def:gapped}.
    To handle the remaining coordinates, we have by assumption 
    \[
    \|\nabla H_N(\bsig^*)\|_{L^2}
    =
    \|\bG^{\sym} \bsig^*\|_{L^2}
    \leq 
    3\sqrt{N}.
    \]
    This implies any subset of at most $\delta N$ coordinates of $\nabla H_N(\bsig^*)$ has $L^1$ norm at most $O(\sqrt{\delta}N)$.
    Therefore 
    \[
    \la
    \nabla H_N(\bsig^*)
    ,
    \bsig^*-\bsig
    \ra
    \geq 
    (\rho-\delta)\gamma N - O(\sqrt{\delta}N)
    \geq 
    \rho\gamma N-O(\sqrt{\delta}N).
    \]
    For the other term, direct application of Proposition~\ref{prop:restricted-norm-bound} yields 
    \begin{equation}
    \label{eq:quadratic-bound}
    \la (\bsig^*-\bsig),\bG^{\sym}(\bsig^*-\bsig)\ra
    \leq 
    O\big(N\sqrt{\rho^3 \log(1/\rho)}\big).
    \end{equation}
    Since we chose parameters such that $\delta\ll\rho\ll\gamma$, we conclude as desired that 
    \[
     H_N(\bsig^*)-H_N(\bsig)
     \geq 
     \rho\gamma N/2.
     \qedhere
    \]
\end{proof}

\begin{proof}[Proof of Theorem~\ref{thm:main}]
    Choose $\beta$ large depending on $(\gamma,\rho,\delta)$ and define the Hamming sphere and ball
    \[
    S_{\rho}(\bsig^*)
    =
    \{\bsig\in\bSig_N~:~\|\bsig-\bsig^*\|_{L^1}=\rho N\},
    \quad\quad
    B_{\rho}(\bsig^*)
    =
    \{\bsig\in\bSig_N~:~\|\bsig-\bsig^*\|_{L^1}\leq\rho N\}.
    \]
    Then on the high-probability events of Proposition~\ref{prop:DGZ} and Lemma~\ref{lem:sphere-energies-smaller}, we have
    \[
    \frac{\mu_{\beta}(S_{\rho}(\bsig^*))}{\mu_{\beta}(B_{\rho}(\bsig^*))}
    \leq 
    \frac{\mu_{\beta}(S_{\rho}(\bsig^*))}{\mu_{\beta}(\bsig^*)}
    \leq 
    2^N e^{-\beta \rho\gamma N/2}
    \leq 
    e^{-c\beta N}.
    \]
    The latter bound holds since $\beta$ was chosen large.
    Further since $H_N$ is an even function and $\rho<1/2$,
    \[
    \mu_{\beta}(B_{\rho}(\bsig^*))\leq 1/2 \leq 
    \mu_{\beta}(\bSig_N\backslash B_{\rho}(\bsig^*)).
    \]
    Cheeger's inequality applied to the set $B_{\rho}(\bsig^*)$ thus implies the relaxation time of Glauber dynamics is exponentially large, completing the proof (see e.g.\ \cite[Chapter 12]{levin2017markov}).
\end{proof}

\begin{remark}
    Theorem~\ref{thm:main} extends to Wigner matrices with IID entries and uniformly bounded fourth moments.
    Indeed Proposition~\ref{prop:DGZ} continues to ensure the existence of gapped states under these conditions, while the fourth moment bound ensures that $\|\bG^{\sym}\|_{\op}\leq 3$ in Lemma~\ref{lem:sphere-energies-smaller} has high probability \cite{BaiYin}.
    Otherwise, the proof above used Gaussianity only in Proposition~\ref{prop:restricted-norm-bound}, which is proved via union bound over $I\subseteq [N]$.
    However a Lindeberg interpolation argument exactly as in \cite[Section 4]{chatterjee2005simple} shows that 
    \[
    \sup_{\substack{\bx\in \{-1,0,1\}^N\\ \|\bx\|_{L^1}\leq \rho N}}
    |\la \bx,\bG^{\sym}\bx\ra|
    \]
    is universal in the disorder under the moment conditions of e.g.\ Proposition~\ref{prop:DGZ}. 
    This implies universality of the relevant estimate \eqref{eq:quadratic-bound} and hence Theorem~\ref{thm:main}.
\end{remark}

\subsection*{Acknowledgement}

Thanks to Yatin Dandi, David Gamarnik, Anouar Kouraich, Mehtaab Sawhney, and Lenka Zdeborov{\'a} for helpful comments.

\bibliographystyle{alphaabbr}
\bibliography{bib}

\end{document}

%% file: commands.tex
\theoremstyle{definition}

\newtheorem*{rmk*}{Remark}

\def\la{\langle}
\def\ra{\rangle}

\def\bbR{{\mathbb{R}}}
\def\bbE{{\mathbb{E}}}

\def\bbP{{\mathbb{P}}}

\def\bG{{\boldsymbol G}}

\def\bx{{\boldsymbol x}}

\def\bsig{{\boldsymbol \sigma}}

\def\op{{\mathsf{op}}}

\def\sym{{\mathsf{sym}}}

\def\mix{{\mathsf{mix}}}
\def\tmix{{t_{\mix}}}

\def\bSig{{\boldsymbol{\Sigma}}}

\def\beq#1\eeq{%
    \begin{equation}%
    #1%
    \end{equation}%
}

\def\baln#1\ealn{%
    \begin{align*}%
    #1%
    \end{align*}%
}
\def\balnn#1\ealnn{%
    \begin{align}%
    #1%
    \end{align}%
}